\documentclass[12pt, reqno]{amsart}

\usepackage{amsmath}
\usepackage{amssymb}
\usepackage{amsfonts}
\usepackage{graphicx}
\usepackage{amsthm}
\usepackage{enumerate}
\usepackage[mathscr]{eucal}
\usepackage{lscape}
\usepackage{dsfont}
\usepackage{color}
\usepackage{mathtools}

\usepackage{setspace}
\newtheorem{theor}{Theorem}[section]

\newtheorem{rem}{Remark}
\newtheorem{prop}[theor]{Proposition}

\newtheorem{lem}[theor]{Lemma}
\newtheorem*{guess}{Conjecture}

\newcommand{\s}{Szeg\H{o} kernel}
\newcommand{\CP}{\mathds{C}\mathrm{P}}
\newcommand{\de}{\partial}

\def\bC{{\mathbb C}}

\newcommand{\KE}{K\"{a}hler-Einstein\ }

\newcommand{\R}{\mathbb{R}}

\newcommand{\C}{\mathbb{C}}
\newcommand{\N}{\mathbb{N}}

\newcommand{\K}{K\"{a}hler\ }

\begin{document}

\title[A characterization of complex space forms]{A characterization of complex space forms via Laplace operators}
\author{Andrea Loi, Filippo Salis, Fabio Zuddas}
\address{Dipartimento di Matematica e Informatica, Universit\`a di Cagliari\\Via Ospedale 72, 09124 Cagliari (Italy)}
\email{loi@unica.it, filippo.salis@gmail.com, fabio.zuddas@unica.it}

\thanks{}
\subjclass[2010]{} 

\begin{abstract} Inspired by the work of Z. Lu and G. Tian \cite{lutian}, in this paper we address the problem of studying those \K\ manifolds  
satisfying the $\Delta$-property, i.e. such that  on a neighborhood of each of its points  the $k$-th power of the  \K Laplacian is a polynomial function  of the  complex Euclidean Laplacian,
for all positive integer $k$ (see below for its definition).
We prove two results: 1. if  a \K\ manifold  satisfies the $\Delta$-property then its curvature tensor is parallel; 2. if  an Hermitian symmetric space of classical type  satisfies  the $\Delta$-property then it is a  complex space form  (namely it has constant holomorphic sectional curvature).
In view of these results we believe that if a complete and simply-connected  \K\ manifold  satisfies the $\Delta$-property then it is a complex space form.

\end{abstract}
\maketitle
\tableofcontents

\section{Introduction and statement of the main results}

Let $\Delta$ be the \K Laplacian on an $n$-dimensional \K manifold $(M, g)$ i.e., in local coordinates $z = (z_1, \dots, z_n)$, 
$$\Delta=\sum_{i,j=1}^n g^{i\bar j}\frac{\de^2}{\de z_j\de\bar z_i},$$
where $ g^{i\bar j}$ denotes the inverse matrix of the \K metric. We define the complex Euclidean Laplacian with respect to $z$ as 
the differential operator
$$\Delta_c^z= \sum_{i=1}^n\frac{\de^2}{\de z_i\de\bar z_i}.$$ 

A key point in  Lu and Tian's proof of the local rigidity theorem (\cite{lutian} Theor. 1.2) supporting their conjecture about the characterization of the Fubini-Study metric $g_{FS}$ on $\CP^n$ through the vanishing of the log-term of the universal bundle, consists in a special relation between \K and complex Euclidean Laplacians that occurs on $(\CP^n, g_{FS})$.
More precisely, Lu and Tian have shown the following:
\begin{theor}[\cite{lutian} Prop. 6.1]\label{cpn}
In the center $x_0$ of the affine coordinate system $z$ of $(\CP^n, g_{FS})$,  every  smooth function $\phi$ defined in a neighborhood of $x_0$  fulfills   the following equations for every positive integer $k$
 \begin{equation}\label{affine}
 \Delta^k\phi(0)=p_k(\Delta_c^z)\phi(0),
 \end{equation}
where $p_k$ is a monic polynomial of degree $k$ with real coefficients (and consequently constant term equal to zero).
\end{theor}

Since complex projective spaces are homogeneous manifolds, the theorem actually states that for $M = (\CP^n, g_{FS})$ the following property holds:

\bigskip

{\em {\bf ($\Delta$-property)} For any arbitrary point  $x \in M$ there exists a coordinate system $z$ centered at $x$ such that \eqref{affine} holds for any positive integer $k$, being $p_k$ a monic polynomial of degree $k$ independent of $x$ with real coefficients.}

\bigskip

The $\Delta$-property is trivially verified also for $\C^n$ endowed with the flat \K metric $g_0$ and for the complex hyperbolic space $\C H^n = \{ z \in \C^n \ | \ \sum_{i=1}^n |z_i|^2 < 1 \}$ endowed with the \K metric $g_{hyp}$ whose associated \K form is $\omega_{hyp} = -\sqrt{-1} \partial \bar \partial \log (1 - \sum_{i=1}^n |z_i|^2)$. This immediately follows again by homogeneity and from the following: 

\begin{prop}\label{radial}
Condition \eqref{affine} is satisfied for any positive integer $k$ in the center of a radial metric\footnote{Namely a \K metric admitting a \K potential which depends only on the sum $|z|^2 = |z_1|^2 + \mathellipsis+ |z_n|^2$ of the moduli of a local coordinates’ system $z$.}.
\end{prop}

This proposition generalizes Theorem \ref{cpn}. Its proof (which follows the same outlines of Lu and Tian's result) will be given in the next Section. 

By the above considerations it is then natural to  try to classify those \K\ manifolds satisfying the $\Delta$-property.
In this direction we have the following two theorems.

\begin{theor}\label{mainteor1}
Let $(M,g)$ be a \K  manifold which satisfies the $\Delta$-property. Then its curvature tensor is parallel.
\end{theor}

\begin{theor}\label{mainteor2}
An Hermitian symmetric space of classical type  satisfying the $\Delta$-property is a complex space form.
\end{theor}

If  Theorem  \ref{mainteor2} could be  extended also to  the case of Hermitian symmetric spaces of   {\em exceptional types} 
then  the complex space forms should be characterized by the $\Delta$-property as expressed by the following:

\begin{guess}
The only complete and simply-connected   \K\ manifolds satisfying  the $\Delta$-property are the complex space forms.
\end{guess}

The following two sections are devoted to the proofs of Proposition \ref{radial}, Theorem \ref{mainteor1} and Theorem \ref{mainteor2}.

\section{Proofs of Proposition \ref{radial} and Theorem \ref{mainteor1}}

\noindent{\it Proof of Proposition \ref{radial}.}  Let $\Phi(|z|^2)$ be a \K potential of a radial \K metric $g$. Then the inverse matrix of $g$ reads locally as
 $$g^{i\bar j}=\frac{1}{\Phi'}\left(\delta_{ij}-\frac{\Phi''}{\Phi'+|z|^2\Phi''}z_j\bar z_i \right),$$
 where $\Phi'$ and $\Phi''$ represent  the first and the second derivative of $\Phi$ with respect to $t=|z|^2$.

We immediately see that \eqref{affine} is verified at point $z=0$ for $k=1$ (up to rescaling $z$). Now, let us assume  by induction that \eqref{affine} holds true for some $k$ and let us prove it for $ k + 1$.
By the inductive assumption one has
$$\Delta^{k+1}\phi(0)=\sum_{l=0}^ka_{k,l}(\Delta_c^z)^l(\Delta \phi)(0)=$$
\begin{equation}\label{recursive}
=\sum_{l=0}^ka_{k,l}(\Delta_c^z)^l\left( \frac{1}{\Phi'}\sum_{i,j}\left(\delta_{ij}-\frac{\Phi''}{\Phi'+|z|^2\Phi''}z_j\bar z_i \right)\frac{\de^2\phi}{\de z_j\de \bar z_i} \right)(0).
\end{equation}
In order to prove that
\begin{equation}\label{recursiveK+1}
\Delta^{k+1}\phi(0) = \sum_{l=0}^{k+1}a_{k+1,l}(\Delta_c^z)^l \phi(0)
\end{equation}
it is enough to show this for $\phi = |z^P |^2$ and $\phi = z^P \bar z^Q$ with $ P \neq Q$, where we are denoting $ z^P = z_1^{P_1}\dotsb z_n^{P_n}$  for  $P = (P_1,\mathellipsis, P_n) \in \N^n$, since every smooth $\phi$ decomposes as a series of such monomials.

Let us first consider the case $\phi=z^P\bar z^Q$, $P\neq Q$. On the one hand, one has $(\Delta_c^z)^l\phi(0)=0$; on the other hand, by using  \eqref{recursive}, it is easy to see that $\Delta^{k+1}\phi(0)=0$, and then (\ref{recursiveK+1}) is trivially true for any choice of the coefficients $a_{k+1,l}$.

 Now, consider $\phi = |z^P|^2$ and let $|P|:=  P_1 + \mathellipsis + P_n=p$. In this case, one has that (\ref{recursiveK+1}) is true if and only if $\Delta^{k+1}\phi(0) = p! P! a_{k+1,p}$. By using \eqref{recursive}, one has
\begin{equation*}
\Delta^{k+1}\phi(0)=\sum_{l=0}^ka_{k,l}(\Delta_c^z)^l\left(\sum_{i} \frac{P_i^2|z^{P-e_i}|^2}{\Phi'}-\sum_{i,j}\frac{\Phi''P_iP_j|z^P|^2}{\Phi'(\Phi'+|z|^2\Phi'')} \right)\Big|_0,
\end{equation*}
where we are denoting by $e_i$ the vectors of the canonical basis of $\R^n$.
By taking into account  Leibniz's rule for the derivatives of a product of functions, we can state that for any $P\in \N^n$  there exists a constant $C_{p,l}^\psi$ depending only on $p$, $l$ and the radial smooth function $\psi$ such that 
$$(\Delta^z_c)^l\Big(|z^P|^2\psi(|z|^2)\Big)\Big|_0 =C_{p,l}^\psi (\Delta^z_c)^p|z^P|^2\Big|_0=C_{p,l}^\psi  p!P!.$$
Furthermore,  if $p>l$, then  $C_{p,l}^\psi=0$ independently of $\psi$ and $C_{h,h}^{\Phi'^{-1}}=1$ for every positive  integer $h$ because we rescaled local coordinates $z$ so that $\Phi'(0)=1$.
After a straightforward computation, we get the following relation which determines the polynomial $p_{k+1}$:
$$a_{k+1,p} = a_{k,p-1}+\sum_{l=p}^k a_{k,l} \left( C_{p-1,l}^{\Phi'^{-1}}-p^2 C_{p,l}^{\frac{\Phi''}{\Phi'(\Phi'+|z|^2\Phi'')}} \right).$$
\hfill $\Box$

The aim of the rest of the section is to prove Theorem \ref{mainteor1}, i.e. that if a \K manifold satisfies the $\Delta$-property then the covariant derivatives of its Riemann tensor $R$ vanish identically. 

We begin by showing the following result (which will be used in the proofs of both Theorem \ref{mainteor1} and Theorem \ref{mainteor2}).

\begin{theor}\label{ke}
A \K  manifold $(M,g)$ is Einstein if and only if for each point $x \in M$ there exist local coordinates $z$ centered at $x$ such that (\ref{affine}) is satisfied for $k=1,2$.
\end{theor}
\begin{proof}
Let $x \in M$ and $z$ be a holomorphic normal coordinate system on $M$ centered at $x$. Clearly, $\Delta \phi(0) = \Delta_c^{ z} \phi(0)$ and (\ref{affine}) is satisfied for $k=1$. Now, if $(M,g)$ is a Einstein manifold, in local coordinates we have\footnote{We are going to use  the notation $\partial_i$ to denote $\frac{\partial}{\partial z_i}$ and a similar notation for higher order derivatives. We are also going to use Einstein's summation convention for repeated indices.}
 \begin{equation}\label{einstein}
 \lambda  g_{i\bar j}=\textrm{Ric}_{i\bar j}= g^{k \bar h}\left( - \partial_{ k \bar h}  g_{i\bar j}+ g^{p \bar q}\partial_k   g_{i\bar q}\partial_{\bar h}   g_{p\bar j}\right).
 \end{equation}
Hence, if we evaluate the previous equation at $x$, we get
\begin{equation}\label{sumder2}
 \sum_{h}\partial_{h\bar h}  g^{i \bar j}(0)=\lambda \delta^{i j}.  
 \end{equation}
By \eqref{sumder2}, we get
 \begin{equation}\label{laplquad}
 \Delta^2  \phi(0)=   g^{h\bar k} \partial_{k\bar h}\big(  g^{i\bar j} \partial_{j \bar i}\phi \big) \Big|_0=\Big((\Delta_c^{ z})^2+\lambda \Delta_c^{ z}\Big)\phi(0)
 \end{equation}
that is (\ref{affine}) is satisfied also for $k=2$.

 Conversely, let us now suppose that for each $x \in M$ there exists a local coordinate system $w$ with respect to which \eqref{affine} is fulfilled for $k=1,2$.  By comparing in both sides of \eqref{affine} for $k=2$ the third order derivative's  coefficient, we get 
 $$\frac{\de  g^{i\bar j}_w}{\de w_k}(0)+\frac{\de  g^{k\bar j}_w}{\de w_i}(0)=0$$
 for every index $i$, $j$ and $k$, where we denote by $g^{i\bar j}_w$ the $(i,\bar j)$ entry of the inverse matrix of $g\left(\frac{\de}{\de w_\alpha},\frac{\de}{\de \bar w_\beta}\right)$. Let $z$ be a holomorphic normal coordinate system around the same point of $w$. Therefore, the previous equation implies 
 $$\frac{\de z_\gamma}{\de w_k}\frac{\de}{\de z_\gamma}\left(\overline{\frac{\de w_i }{\de z_\alpha}} g^{\alpha\bar \beta}_z \frac{\de w_j}{\de z_\beta} \right)\Big|_0+\frac{\de z_\gamma}{\de w_i}\frac{\de}{\de z_\gamma}\left(\overline{\frac{\de w_k }{\de z_\alpha}} g^{\alpha\bar \beta}_z \frac{\de w_j}{\de z_\beta} \right)\Big|_0=0.$$
Since \eqref{affine} for $k=1$ reads as
$$\Delta_c^z\phi(0)=\Delta\phi(0)= \Delta_c^{ w}\phi(0)= \sum_{i,j,\alpha,\beta}\frac{\de z_\alpha}{\de  w_i}\overline{\frac{\de  z_\beta}{\de {  w}_i}}\frac{\de^2\phi}{\de z_\alpha\de\bar {z}_\beta}\Big|_0,$$
 $\frac{\de z_\alpha}{\de  w_i}\Big|_0$ needs to be a unitary matrix.
Hence, we get
$$ \frac{\de^2 w_j}{\de z_\alpha\de z_\beta}(0)  =0$$
 for every index $j$, $\alpha$ and $\beta$.
 By considering that  local coordinates $w$ satisfy \eqref{laplquad}, hence they need to satisfy \eqref{sumder2}, we have
 $$\lambda \delta^{i j}= \sum_{h}\frac{\de^2g^{i \bar j}_w}{\de w_h\de\bar w_ h}  \Big|_0= \sum_{h}\frac{\de z_\gamma}{\de w_h}\frac{\de w_j}{\de  z_\beta}\frac{\de^2 g^{\alpha\bar \beta}_z}{\de z_\gamma \de\bar z_\delta} \overline{\frac{\de w_i}{\de z_\alpha}}\ \overline{\frac{\de z_\delta}{\de w_h}}\Big|_0=\sum_h \frac{\de^2 g^{i\bar j}_z}{\de z_h \de z_h}\Big|_0.$$
This means that $\textrm{Ric}_{i\bar j} = \lambda g_{i\bar j}$ at $x$, and then by the arbitrariness of $x$ we conclude that $g$ is a \KE metric. The theorem is proved.
\end{proof}

\begin{rem}
\rm Notice that, by combining Theorem \ref{ke} with the uniformization theorem, one gets a proof of our conjecture for complex dimension $n=1$.
\end{rem}

Now we are finally ready to prove Theorem \ref{mainteor1}.

\vspace{0.4cm}

\noindent{\it Proof of Theorem \ref{mainteor1}. } For any $x \in M$, let $w$ be  a local coordinate system centered at $x$ with respect to which (\ref{affine}) is satisfied for any positive integer $k$. In particular, by comparing the fifth order derivative's  coefficients of both sides of (\ref{affine}) for $k=4$, we get that
 \begin{equation}\label{5der}
 \left(\de_{h\bar k l}  g^{i\bar j}_w+\de_{i\bar k l}  g^{h\bar j}_w+\de_{i\bar k h}  g^{l\bar j}_w+\de_{h\bar j l}  g^{i\bar k}_w+\de_{i\bar j l}  g^{h\bar k}_w+\de_{i\bar j h}  g^{l\bar k}_w\right)(0)=0
 \end{equation}
 for every choice of indexes $i$, $j$, $h$, $k$ and $l$.

 Let $z$ be a holomorphic normal coordinate system around the same point of $w$.  By taking into account that in the proof of Theorem \ref{ke} we have shown that every second order derivative of the holomorphic change of coordinates sending $z$ to $w$ vanishes at $z=0$ and  $\frac{\de z_\beta}{\de w_j}\Big|_0$ is a unitary matrix, 
 we get 
 $$\sum_{\alpha,\beta}\frac{\de^3 g^{\alpha \bar \beta}_z}{\de z_\gamma\de\bar z_\delta\de z_\epsilon} \Big(\overline{\frac{\de z_\beta}{\de w_j}}\ \overline{ \frac{\de z_\delta}{\de w_k}}+\overline{\frac{\de z_\beta}{\de w_k}}\ \overline{ \frac{\de z_\delta}{\de w_j}} \Big)\Big( \frac{\de z_\alpha}{\de w_i}\frac{\de z_\gamma}{\de w_h}\frac{\de z_\epsilon}{\de w_l}+$$
 $$+\frac{\de z_\alpha}{\de w_h}\frac{\de z_\gamma}{\de w_i}\frac{\de z_\epsilon}{\de w_l}+\frac{\de z_\alpha}{\de w_l}\frac{\de z_\gamma}{\de w_i}\frac{\de z_\epsilon}{\de w_h} \Big)\Big|_0=0. $$
Therefore,  for every index $\alpha$, $\beta$, $\gamma$, $\delta$ and $\epsilon$, a relation similar to \eqref{5der} holds true also with respect to holomorphic normal coordinates $z$:
$$ \left(\de_{\gamma\bar \delta \epsilon}  g^{\alpha\bar \beta}_z+\de_{\alpha\bar \delta \epsilon}  g^{\gamma\bar \beta}_z+\de_{\alpha\bar \delta \gamma}  g^{l\bar \beta}_z+\de_{\gamma\bar \beta \epsilon}  g^{\alpha\bar \delta}_z+\de_{\alpha\bar \beta \epsilon}  g^{\gamma\bar \delta}_z+\de_{\alpha\bar\beta \gamma}  g^{\epsilon\bar \delta}_z\right)(0)=0.$$
Since 
$$\frac{\de^3  g^{\alpha\bar \beta}}{\de z_\gamma \de\bar z_\delta \de z_\epsilon}(0)=-\frac{ \de^3 g_{\alpha\bar \beta}}{\de z_\gamma\de\bar z_\delta \de z_\epsilon}(0),$$
we get 
$$\frac{\de^3  g_{\alpha\bar \beta}}{\de z_\gamma \de\bar z_\delta \de z_\epsilon}(0)=0$$
for every index $\alpha$, $\beta$, $\gamma$, $\delta$ and $\epsilon$. It follows that the  covariant derivatives of the Riemann tensor must vanish identically as desired. \hfill $\Box$

\begin{rem}
\rm Notice that in fact we have proved the stronger statement that if for any $x \in M$ there exists a local coordinate system centered at $x$ with respect to which (\ref{affine}) is satisfied for $k=1,\dots,4$, then curvature tensor of $M$ is parallel.
\end{rem}

\section{Proof of Theorem \ref{mainteor2}}

The first step in the proof of our second main result consists in characterizing  complex projective spaces among  Hermitian symmetric spaces of compact type\footnote{From now on HSSCT.} by means of the relations between \K and complex Euclidean Laplacians \eqref{affine}. Before proving such theorem, we need  a technical lemma.

We have  to recall that Alekseevsky and Perelomov described explicitly in \cite{AlPer}  holomorphic coordinates  for every flag manifold $G/K$ which turns out to be an orbit of the adjoint action of a classical compact semisimple Lie group $G$ on its Lie algebra $\mathfrak{g}$.
Throughout the paper we are going to call these coordinates \emph{Alekseevsky-Perelomov coordinates}.
Moreover, we will write $\CP^1_r$ to denote the product of $r$ complex projective spaces $\CP^1$  equipped with  the product metric $g^r_{FS}=g_{FS}\oplus \mathellipsis\oplus g_{FS}$.

\begin{lem}\label{APcoord}
Let  $M$ be an irreducible classical $n$-dimensional  HSSCT of rank $r$ endowed the (unique up to rescaling) \KE metric.
 Then there exists a local coordinate system $w$ such that the \K immersion's equations of  $\big(\CP^1_r,  g^r_{FS}\big)$ into $M$ read  as 
 $$\begin{cases}
 w_i=z_i &\text{for }i=1,\mathellipsis,r\\
 w_i=0 &\text{for }i=r+1,\mathellipsis,n
 \end{cases}$$
 where  $z$ are  affine  coordinates on  $\CP_r^1$.
\end{lem}
\begin{proof}
We are going to show explicitly in the following case by case analysis  how the isometric embedding of  $\big(\CP_r^1, g_{FS}^r\big)$ into an irreducible classical HSSCT of rank $r$ reads with respect to  Alekseevsky-Perelomov coordinates. The desired local coordinate system $w$ is obtained in a obvious way from the Alekseevsky-Perelomov one.

\bigskip
\noindent
\textbf{Case 1}: ${SU(N)}/{S\big(U(k) \times U(N-k)\big)}$.\\
The Alekseevsky-Perelomov coordinates are given by the entries $w_{ij}$ of a complex $(N-k) \times k$ matrix $W$ and the potential of the \KE metric (up to a constant) with respect to these coordinates reads as
\begin{equation*}\label{potentialGrass}
 \log \det\left( I_k + {}^T\bar W W \right).
\end{equation*}
We can easily prove that a \K  immersion of $\bC P^1_m$, where $m=\min\{ k, N-k \}$, is given by sending affine coordinates $z = (z_1, \dots, z_m)$ to the matrix $W(z) \in M_{N-k, k}(\bC)$ defined by
$$W(z)_{ij} = z_i \delta_{ij}.$$
 
\bigskip
\noindent
\textbf{Case 2}: ${SO(2N)}/{U(N)}$.\\
The Alekseevsky-Perelomov coordinates are given by the entries $w_{ij}$ of a skew-symmetric complex $N \times N$ matrix $W$ and the potential of the \KE metric (up to a constant) with respect to these coordinates reads as
\begin{equation*}\label{potentialSONpariU}
 \log \det\left( I_N + {}^T\bar W W\right).
\end{equation*}
We can easily prove that  a \K immersion of $\bC P^1_{ \left[ \frac{N}{2}\right]}$  is given by sending affine coordinates $z = (z_1, \dots, z_{\left[ \frac{N}{2}\right]})$ to the matrix
\begin{small}
$$W(z)= \left( \,\begin{array}{cccccccc}
0 &   z_1     & 0 & 0   &\cdots \\
-z_1  & 0 & 0  & 0   & \cdots\\
 0  &     0   & 0 & z_2   &  \\
0 & 0 & -z_2 & 0   &\ddots  \\

 \vdots & \vdots  &  &\ddots &  \ddots   \\
\end{array}\,\right).$$
\end{small}

\bigskip
\noindent
\textbf{Case 3}: ${Sp(N)}/{U(N)}$.\\
The Alekseevsky-Perelomov coordinates are given by the entries $w_{ij}$ of a symmetric $N \times N$ complex matrix $W$ and the potential the \KE metric (up to a constant) in these coordinates reads as
\begin{equation*}\label{potentialSPN}
 \log \det\left( I_N + {}^T\bar W W\right).
\end{equation*}
We can easily prove that  a \K immersion of $\bC P^1_N$  is given by sending affine coordinates $z = (z_1, \dots, z_{N})$ to the matrix
$$W(z)_{ij}=z_i\delta_{ij}.$$ 

\bigskip
\noindent
\textbf{Case 4a}: ${SO(2N)}/\big({SO(2N-2) \times SO(2)}\big)$, with $N \geq 4$.\\
The Alekseevsky-Perelomov coordinates are $(v_2, \dots, v_N, v_2', \dots, v_N') \in \bC^{2N-2}$ and the potential the \KE metric (up to a constant) in these coordinates reads as
\begin{equation*}\label{potentialSONpari}
 \log\left(1 + \sum_{j=2}^N |v_j|^2 + \sum_{j=2}^N |v'_j|^2  + 4 \left|  \sum_{j=2}^N v_j v'_j\right|^2 \right).
\end{equation*}
We can easily prove that  a \K immersion of $\bC P^1 \times \bC P^1$ is given by setting
$$v_2 = \frac{z_1}{\sqrt2}, v_3 = \frac{z_2}{\sqrt2}, v_4 = \mathellipsis = v_N=0,$$
$$v'_2 = 0, v'_3 = \frac{z_1}{\sqrt2}, v'_4 = \frac{z_2}{\sqrt2}, v'_5 =\mathellipsis = v'_N=0,$$
where $(z_1,z_2)$ are affine coordinates on $\bC P^1 \times \bC P^1$.

\bigskip
\noindent
\textbf{Case 4b}: ${SO(2N+1)}/\big({SO(2N-1) \times SO(2)}\big)$, with $N \geq 4$.\\
The Alekseevsky-Perelomov coordinates are $(v_2, \dots, v_N, v'_2, \dots, v'_N, u) \in \bC^{2N-1}$. The potential the \KE metric (up to a constant) in these coordinates reads as
\begin{equation*}\label{potentialSONdispari}
 \log\left(1 + \sum_{j=2}^N |v_j|^2 + \sum_{j=2}^N |v'_j|^2 + |u|^2  + 4 \left|  \sum_{j=2}^N v_j v'_j\ - u^2 \right|^2 \right)
\end{equation*}
and a \K  immersion of $\bC P^1 \times \bC P^1$ is given by setting
$$v_2 = \frac{z_1}{\sqrt2}, v_3 = \frac{z_2}{\sqrt2}, v_4 =\mathellipsis= v_N=0,$$ 
$$v'_2 = 0, v'_3 = \frac{z_1}{\sqrt2}, v'_4 = \frac{z_2}{\sqrt2}, v'_5 =\mathellipsis= v'_N=u=0,$$
where $(z_1,z_2)$ are affine coordinates on $\bC P^1 \times \bC P^1$ .
\end{proof}

We are now in a  position to characterize complex projective spaces among irreducible HSSCT.

\begin{theor}\label{hssct}
Complex projective spaces are the unique classical irreducible HSSCT which satisfy the $\Delta$-property.
\end{theor}
\begin{proof}
Let $M$ be an $n$-dimensional HSSCT endowed with the \KE metric $ g$.  We denote by $\lambda$ the Einstein constant. Let $\tilde z$ be a holomorphic normal coordinate system. 

By combining \eqref{einstein}, \eqref{sumder2} and  \eqref{laplquad} above, we get that  every smooth function $\phi$ defined in a neighborhood $V$ of the center of $\tilde z$ fulfills the following
$$\Delta^3  \phi(0) =\Big((\Delta_c^{\tilde z})^3 +3\lambda(\Delta_c^{\tilde z})^2+\lambda^2\Delta_c^{\tilde z}\Big) \phi(0)+2 \sum_{ l,h=1}^{n} \partial_{l\bar h} g^{i\bar j}\partial_{j  h \bar l\bar i}\phi\ \Big|_0+$$
\begin{equation}\label{laplcube}
 +\sum_{ l,h=1}^{n} \partial_{lh} g^{i\bar j}\partial_{j \bar h \bar l\bar i}\phi\ \Big|_0+\sum_{ l,h=1}^{n} \partial_{\bar l\bar h}  g^{i\bar j}\partial_{j  h  l\bar i}\phi\ \Big|_0+\sum_{l,h=1}^{n}\partial_{l h\bar l\bar h} g^{i \bar j} \partial_{j\bar i}\phi\ \Big|_0,
 \end{equation}
 where we use that by differentiating  \eqref{einstein} and evaluating in the center of $\tilde z$, the coefficients of the  third order derivatives of $\phi$ vanish.

If $z$ are  affine coordinates  on  $(\CP_r^1, g_{FS}^r)$, where $r$ is equal to the rank of $M$, by taking into account Lemma \ref{APcoord} and by considering that  Alekseevsky-Perelomov coordinates $w$ are  normal  up to rescaling by suitable  constants that we call $\mu_i$  (see \cite{bochnerflag} Theor. 1), we can  compute
\begin{equation}\label{comp1}
\Delta^3  \big(|z_1|^4\big)\Big|_0=\frac{3\lambda}{(\mu_1)^2}\big(\Delta_c^{w}\big)^2(|z_1|^4)\Big|_0+\frac{8}{(\mu_1)^2}\frac{\partial^2  g^{1\bar 1}}{\partial w_1\partial\overline{ w}_1}\Big|_0=\frac{12\lambda+16}{(\mu_1)^2}.
\end{equation}
Furthermore, if $r\neq 1$, namely $M$ is different from a complex projective space, we also compute
$$\Delta^3   \big(|z_1z_2|^2\big)\Big|_0=\frac{3\lambda}{\mu_1\mu_2}\big(\Delta_c^{w}\big)^2  (|z_1z_2|^2)\Big|_0+$$
\begin{equation}\label{comp2}
+4\left(\frac{1}{(\mu_1)^2} \frac{\partial^2 g^{2\bar 2}}{\partial w_1\partial\overline{ w}_1}+\frac{1}{(\mu_2)^2} \frac{\partial^2 g^{1\bar 1}}{\partial w_2\partial\overline{ w}_2}+\frac{1}{\mu_1\mu_2} \frac{\partial^2 g^{1\bar 2}}{\partial w_2\partial\overline{ w}_1}+\frac{1}{\mu_1\mu_2} \frac{\partial^2 g^{2\bar 1}}{\partial w_1\partial\overline{ w}_2} \right)=\frac{6\lambda}{\mu_1\mu_2}.
\end{equation}

If $M$ has rank greater than $1$, 
let us assume by contradiction that the $\Delta$-property is valid, in particular around each point of $M$ there exists a local coordinate system with respect to which \eqref{affine} is satisfied for $k=1,2,3$. Let us denote such coordinate system by $f=(f_1,\mathellipsis,f_n)$.

Since we have shown in Theorem \ref{ke} that every second order derivative of the holomorphic change of coordinates sending $f$ to $\tilde z$ vanish at $f=0$,
we get
$$\Delta^{3}\phi(0)=\Big((\Delta_c^{ f})^{3}+\sum_{i=1}^2 a_i(\Delta_c^{ f})^i\Big)\phi(0)=$$
$$=\Big(\sum_{i=1}^2 a_i(\Delta_c^{\tilde z})^i\Big)\phi\Big|_0+\sum_{i_1,i_2,i_{3},\alpha,\beta}\frac{\de^{3} \tilde z_\alpha}{\de { f}_{i_1} \de { f}_{i_2}\de { f}_{i_{3}}}\overline{\frac{\de^{3} \tilde z_\beta}{\de  { f}_{i_1} \de { f}_{i_2}\de { f}_{i_{3}}}}\frac{\de^2\phi}{\de \tilde z_\alpha\de\bar{\tilde z}_\beta}\Big|_0+$$
$$+(\Delta_c^{\tilde z})^{3}\phi\Big|_0+\sum_{\substack{i_1,i_2,i_{3}\\\alpha_1,\mathellipsis,\alpha_{4}}}\frac{\de^{3} \tilde z_{\alpha_{4}}}{\de { f}_{i_1} \de { f}_{i_2}\de { f}_{i_{3}}}\prod_{l=1}^{3} \overline{\frac{\de {\tilde z}_{\alpha_l}}{\de {  f}_{i_l}}}\frac{\de^{4}\phi}{\de\bar{ \tilde z}_{\alpha_1}\de\bar{ \tilde z}_{\alpha_2}\de\bar{\tilde z}_{\alpha_{3}}\de{\tilde z}_{\alpha_{4}}}\Big|_0+ $$
$$+\sum_{\substack{i_1,i_2,i_{3}\\\alpha_1,\mathellipsis,\alpha_{4}}}\overline{\frac{\de^{3} \tilde z_{\alpha_{4}}}{\de { f}_{i_1} \de { f}_{i_2}\de { f}_{i_{3}}}}\ \prod_{l=1}^{3} {\frac{\de {\tilde z}_{\alpha_l}}{\de {  f}_{i_l}}}\frac{\de^{4}\phi}{\de \tilde z_{\alpha_1}\de { \tilde z}_{\alpha_2}\de{\tilde z}_{\alpha_{3}}\de\bar{\tilde z}_{\alpha_{4}}}\Big|_0. $$
The previous formula implies   the relation
$$\Delta^3  \big(|z_1|^4\big)(0)=2\frac{\mu_2}{\mu_1}\Delta^3  \big(|z_1z_2|^2\big)(0),$$ therefore we have a contradiction from the comparison with \eqref{comp1} and \eqref{comp2}.
\end{proof}

\begin{rem}
\rm Notice that we have proved the stronger statement that the complex projective spaces are the unique classical irreducible HSSCT such that around any point there exists a local coordinate system with respect to which (\ref{affine}) is satisfied for $k=1,2,3$.
\end{rem}

\begin{rem}
\rm Theorem \ref{hssct} proves also that  the $\Delta$-property cannot be  satisfied in any not irreducible classical HSSCT, because they contain an embedded $(\CP_2^1,g_{FS}^2)$ whose \K immersion's equations  locally reads with respect to holomorphic normal coordinates as in Lemma \ref{APcoord}. 
\end{rem}

Finally, we can prove our second main result.
\vspace{0.3cm}

\noindent{\it Proof of Theorem \ref{mainteor2}. }Every Hermitian symmetric space\footnote{From now on HSS.} can be decomposed as a  \K product
\begin{equation*}\label{decomposition}
(\C^n,g_0)\times (C_1,g_1)\times\mathellipsis\times( C_h,g_h)\times (N_1, \hat g_1)\times\mathellipsis\times (N_l, \hat g_l),
\end{equation*}
where $(\C^n,g_0)$ is  the flat Euclidean space, $(C_i,g_i)$ are  irreducible HSSCT and $(N_i, \hat g_i)$ are irreducible HSS of noncompact type\footnote{Namely $ N_i$ is a  bounded symmetric domains with a multiple of the Bergman metric denoted by $\hat g_i$.}.

By Theorem \ref{ke}, a  HSS where \eqref{affine} is fulfilled for $k=1,2$, is the flat Euclidean space otherwise  it is a \K product of  HSS of  either compact  or   noncompact type.  Hence, we  are going to prove our statement by characterizing hyperbolic spaces among classical HSS of noncompact type  in analogy with what we have done for projective spaces  in Theorem \ref{hssct}. 

Let $z$ be the restriction of Euclidean coordinates to a classical bounded symmetric domain $(N,\hat g)$
such that $\hat g$ is a \KE metric. Let $(N^*,\hat g^*)$  be its compact  dual.
We can think $z$ as the restriction of Alekseevsky-Perelomov coordinates of $N^*$ to $N$. Furthermore a \K\ potential $\Phi$ for the metric $\hat g$
is given by 
\begin{equation}\label{ff}
\Phi (z, \bar z)=-\Phi^*(z,-\bar z)_{|N}
\end{equation}
where $\Phi^*(z, \bar z)$ is a \K\ potential for $\hat g^*$ (see \cite{symplsymm} and \cite{sympldual} for details).

By \eqref{ff} and   \eqref{laplcube}, we get
$$\Delta_N^3\big (|z_iz_j|^2\big)(0)=-\Delta_{N^*}^3\big (|z_iz_j|^2\big)(0)$$
for every $1\leq i,j\leq \dim(N)$. Hence, if $N$ and $N^*$ are not irreducible or else if they are irreducible but they have rank different from $1$, namely $N$ is not a hyperbolic space, \eqref{affine} for $k=3$ cannot be satisfied as proved in Theorem \ref{hssct}.\hfill $\Box$

\end{document}